\documentclass[10pt]{article}
\usepackage{amsmath,amssymb,amsthm}
\usepackage{graphicx}
\usepackage{amssymb}
\usepackage{graphicx}
\usepackage{amsmath}
\usepackage{amsfonts,amsthm}
\usepackage[english]{babel}
\usepackage[active]{srcltx}
\usepackage[latin1]{inputenc}
\usepackage{epsfig}
\usepackage[all]{xy}
\usepackage{lscape}
\usepackage{amsfonts,amsthm}
\usepackage{amsmath,amssymb,amsthm}
\usepackage{epstopdf}
\newtheorem{theorem}{Theorem}

\newtheorem{definition}{Definition}
\newtheorem{proposition}{Proposition}
\newtheorem{lemma}{Lemma}
\newtheorem{remark}{Remark}

\newcommand{\cC}{{\cal C}}

\newcommand{\cV}{{\cal V}}
\newcommand{\cW}{{\cal W}}

\title{Linear Hyperdoctrines and Comodules}
\author{Mariana Haim and Octavio Malherbe}

\date{December 2016}

\theoremstyle{plain}

\begin{document}

\maketitle
\setlength{\parindent}{0pt}

\begin{abstract}
	\noindent
	In this exposition,  we get examples of what is called a ``linear hyperdoctrine", based on categories of comodules indexed by coalgebras, as in \cite{kn:GP}. This structures 
	can model first order linear logic.
	
\end{abstract}

\section{Introduction}

The aim of this work is to present indexed families of models for Linear Logic coming from Coalgebra Theory. 
This paper connect two areas of research: on the one hand, Logic, more precisely, Linear Logic and a consolidate linear logic version of Hyperdoctrine and on the other hand the field of Coalgebras and Comodules Theory. Moreover, we discuss to use this same paradigm to give a model for Linear Polymorphism. \\
\ \\
In logic, when we allow variables to range over subsets we are dealing with second-order logic. In second order lambda calculus there are two kind of lambda abstractions: over term variables and over type variables, i.e., types also appear as parameters, the idea of variable types, which may be explicitly abstracted and then computed by evaluating with some concrete types.\\
\ \\
We consider the more general framework of linear lambda calculus which is based on intuitionistic linear logic. Girard's Intuitionistic Linear Logic~\cite{kn:G} is a resource sensitive logic, i.e., a Logic with resources that are controlled
by an operator named ``bang" and declared by the symbol ``!": formulas must be used exactly once unless the logical
operator ! is invoked allowing a formula to be used as many times as
required. It formalises the requirement that a given logical assumption (or resource) can only be used once: data, in some sense, is devoured by functions.
Moreover, structural rules such as Weakening and Contraction of Gentzen's sequent calculus are
removed and reintroduced in a controlled manner by this logical operator.\\
\ \\
In 1972 Girard and independently in 1974 Reynolds developed a polymorphic version of the lambda calculus. ``The system F", as Girard named, it is based on the idea of variable types: if $\sigma$ and $\alpha$ are type variables, then $\Lambda \alpha. \sigma$ is a type. The idea is that the variable type $\alpha$ is bounded in ``$\Lambda \alpha.\sigma$" (see~\cite{kn:G2}) and it is declared as it were a term in the calculus. From the computer science point of view, the purpose of this symbol is to make explicit common features on the manipulation of programs in which the same algorithm is invoked in different situations. (It is more common to use the symbol $\forall$ ``an universal quantifier" instead of $\Lambda$ to denote this abstraction at the level of types, as parameters in type expressions). Types appears as parameters not only in other types but also in terms: if we consider expressions like the identity function $\lambda x: A.x$, the instruction given by the term, in some sense, is completely independent of the parameter $A$. We can stress this by writing a term $\Lambda A. \lambda x: A.x.$ Therefore, we have two different abstractions: at the level of term variable and also at the type variable denoted by two kinds of variables.
Of course, this implies that we also have two kinds of evaluations of the expression at both levels. In the same way as the expression $\lambda x: A.x$ of our calculus has type $A\rightarrow A$ we have a symbol to denote the type of the new abstraction $\Lambda A. \lambda x:A.x$ which will be denoted by $\forall A:A \rightarrow A$.\\
\ \\
Therefore, lambda notation provides two differents means for writing expressions which denote two form of abstraction. The simbol $\Lambda X$ operates abstractions over type variable in the same fashion as $\lambda x$. It leads to terms denoting polymorphic functions.
But also, inside terms we have dependency of types, i.e., types as parameters in term:
if $t$ is a term of type $\sigma$ and $\alpha$ is a type variable, then $\Lambda \alpha.t $ is a term of type $\forall \alpha.\sigma$. The variable $\alpha$ is bounded in $\Lambda \alpha .t.$ Moreover, if $t$ is a term of type $\forall \alpha.\sigma$ and $\tau$ is a type , then $t(\tau)$ is a term of type $\sigma (\tau/\alpha)$. Consequently, there are two kind of reduction of abstractions given by: $(\lambda x.t(x))u= t(u/x)$, and $(\Lambda \alpha. t(\alpha))(\tau)=t(\tau/\alpha)$.

Coalgebras are algebraic structures that are partially dual to algebras: the dual vector space of a $\Bbbk$-coalgebra is a $\Bbbk$-algebra
and, although the dualizing functor is not an equivalence of categories, it
does have an adjoint on the right. These two functors
restrict to a contravariant equivalence between finite dimensional
coalgebras and finite dimensional algebras. 
Due to this partial duality, many
definitions in the theory of coalgebras were suggested by the corresponding
concepts for algebras (such as the notion of comodule, the cotensor product of comodules, and the
statements of many results). \\
The partial duality restricts to cocommutative coalgebras (partially dual to the category of commutative algebras). In many senses, the category $Coalg$ of cocommutative coalgebras
is much better than the category of commutative algebras. For example, it is Cartesian closed: the usual tensor product of vector spaces carries a coalgebra structure corresponding to the categorical product, and the set of morphisms from one coalgebra to another carry a natural structure of coalgebra (that ensures closedness of the cartesian structure). Also, it is well known that many categories occuring in algebra are enriched over Coalg. 

Various kind of models have appeared in the literature.
Girard proposed domain theory models: coherence spaces, qualitative domains and stable functions yielding a solid semantic~\cite{kn:G2}.
Associated with Scott domains and continuous functions are several proposals for models of polymorphism such as in \cite{kn:CGW} where types are interpreted as Scott domains and types with variables are continuous functors on a category of Scott domain.

Also, in this direction, Maneggia models (\cite{kn:M}) are based on the domain theoretic characterization given by Coquand, Gunter and Winskel and the central notion of lax limit in partial orders and order preserving maps obtaing sound models of the second order linear type theory developed by Plotkin named $F_{DILL}$ (see \cite{kn:Pl}).
More precisely, lax limits construction gives a right adjoint to the diagonal into the category of functors and lax transformations. This adjunction is what allows the modelling universal quantification. She obtains a linear hyperdoctrine by lifting the structure of a linear category to the category of functors and lax transformations obtaining a linear category easy to index.

Another approach to deal with system $F$ is to consider realisability models, as in \cite{kn:H}. Also, Abramsky and Lenisa (\cite{kn:AyL}) discuss full completeness in models based on the category of Partial Equivalence Relations over a Linear Combinatory Algebra. They introduce the notion of adjoint hyperdoctrine in order to study System F which consists of a co-Kleisli indexed
category of a linear indexed category. 

Related also with the topic but from a more foundational point of view, Reynolds has shown in~\cite{kn:R} how the impredicativity of system F obliges to consider models other than a naive set theoretic interpretation. Just by cardinality considerations this interpretation is impossible to be set theoretical. Pitts has given an alternative presentation of models of polymorphism based on constructive set theory~\cite{kn:Pi}.

A more abstract approach, and close to our point of view, is Seely's work. He introduces a categorical structure for interpreting polymorphic lambda calculus; he considered a linear hyperdoctrine in which the fibred is given by a monoidal linear category~\cite{kn:S2}. In some sense, based on this abstraction, in this paper we give an alternative presentation to the usual "domain order category" presentation of the subject.

A fundamental construction in category theory is the notion of an indexed category. Grunenfelder and Par\'e obtain, in \cite{kn:GP}, sufficient conditions in order that all the ingredients we need work in the field of comodules indexed by coalgebras. The organization of our model is based in their work, in terms of the categorical apparatus of an indexed category captured by them.

A polymorphism for the lambda calculus is understood semantically as an indexed category with some extra conditions.
For this purpose, the intuitionistic linear logic $(\otimes, I, \multimap,!)$ will be enriched with a universal quantification $\forall$ in the sense we mention above, of Girard-Reynolds. A key feature for this construction is the right choice of a subcategory of the category of coalgebras allowing all the conditions fully work. Our contribution is to understand this aspect.
It turns out that cosemisimple coalgebras are the correct choice in order to make all conditions work at the same time. We use this framework to describe the semantic of an hyperdoctrine by choosing to treat cosemisimple coalgebras as if they were the parameter that characterize all the structure.
Specifically, the construction of the model depends on the behavior of an indexed category of comodules that is determined by their parameter when we consider the category of cosemisimple coalgebras. Basically, this is allowed by the following fundamental equivalence: $C$ is cosemisimple if and only if is every $C$-comodule is reducible.
\ \\
\\
It is worth noting the aspects not covered by this paper. Firstly, we concentrate on the categorical aspects of the model construction and we do not readdress the syntax of the calculus itself (see~\cite{kn:Pl},~\cite{kn:M} for a revision). Secondly, the paper focalizes on technical aspects of comodules as a first step to deal with polymorphisms and ignores general properties of its semantics such as soundness, completenes, etc. This paper describes a model of first order logic but the most obvious question that we have not yet addressed is whether this model also accounts for a model of higher order logic with the so called generic object condition. In future work, we are going to deal with this issue.


{\em Aknowledgements:} The authors want to thank Jonas Frey for his valuable comments.


\textit{Organization of the Paper:} In Section $2$ we expose the Preliminaries and fix the notations that will be necessary on Category Theory and Coalgebra Theory. Then, in Section $3$ we present a family of models for Linear Logic based on the category of comodules under suitable coalgebras. We go to the indexed context in Section $4$ and prove, in Section $5$ that it will be a good context for constructing models of Linear Hyperdoctrines. Finally in Section $6$ we discuss conditions about the generic object.

\section{Preliminaries and Notations}

\subsection{Preliminaries on monoidal and cartesian categories}
For general preliminaries and notations on categories we refer to~\cite{kn:McL}. \\
\begin{definition}
	\rm
	A \textit{monoidal} category, also called \textit{tensor} category, is a category $\cV$ with an identity object $I\in\cV$, a bifunctor $\otimes:\cV\times\cV\rightarrow \cV$ and natural isomorphisms $\rho:A\otimes I\stackrel{\cong}\rightarrow A$, $\lambda:I\otimes A\stackrel{\cong}\rightarrow A$, $\alpha:A\otimes (B\otimes C)\stackrel{\cong}\rightarrow (A\otimes B)\otimes C$, satisfying the following coherence commutativity axioms:
	
	\[
	\xymatrix@=14pt{
		A\otimes (I\otimes B) \ar[dr]_{1\otimes\lambda}  \ar[rr]^{\alpha}& & (A\otimes I) \otimes B \ar[ld]^{\rho\otimes 1}  \\
		&  A\otimes B &      }
	\]
	and
	\[
	\xymatrix@=24pt{
		A\otimes (B\otimes (C\otimes D))\ar[d]^{\alpha}\ar[r]^{\alpha}&(A\otimes B)\otimes (C\otimes D) \ar[r]^{\alpha} &  ((A\otimes B)\otimes C)\otimes D \ar[d]^{\alpha}\\
		(A\otimes ((B\otimes C)\otimes D)\ar[rr]_{\alpha}&      &    (A\otimes (B\otimes C))\otimes D}
	\]
	
	A monoidal category $(\cV,\otimes,I,\alpha,\rho,\lambda)$ is said to be \textit{symmetric} if there is a natural isomorphism $\sigma:A\otimes B\stackrel{\cong}\rightarrow B\otimes A$ which satisfies the coherence axioms given by the commutativity of the following diagrams:
	$$
	\begin{array}{cc}
	\xymatrix@=14pt{
		A\otimes B \ar[r]^{\sigma}  \ar[rd]_{id}& B \otimes A \ar[d]^{\sigma}  \\
		&  A\otimes B
	} \hspace{3cm}&
	\xymatrix@=14pt{
		A\otimes I \ar[r]^{\sigma}  \ar[rd]_{\rho}& I\otimes A \ar[d]^{\lambda},  \\
		&  A
	}\end{array}
	\ \
	\xymatrix{
		A\otimes (B\otimes C)\ar[d]^{1\otimes\sigma} \ar[r]^{\alpha}  & (A\otimes B)\otimes C\ar[r]^{\sigma}  & C\otimes (A\otimes B) \ar[d]^{\alpha}  \\
		A\otimes (C\otimes B)\ar[r]^{\alpha} &( A\otimes C)\otimes B \ar[r]^{\sigma\otimes 1} & (C\otimes A)\otimes B.      }
	$$
	
\end{definition}

\begin{definition}
	\rm
	A \textit{closed} monoidal category is a symmetric monoidal category $\cV$ for which each functor $-\otimes B:\cV\rightarrow\cV$ has a right adjoint $[B,-]:\cV\rightarrow\cV$:  $$\cV(A\otimes B,C)\cong\cV(A,[B,C])$$.
\end{definition}

\begin{definition}
	\rm
	A \textit{monoidal functor} $(F,m_{A,B},m_I)$ between monoidal categories $(\cV,\otimes,I,\alpha,\rho,\lambda)$ and $(\cW,\otimes',I',\alpha',\rho',\lambda')$ is a functor $F:\cV\rightarrow\cW$ equipped with:
	\begin{itemize}
		\item[-]
		morphisms $m_{A,B}:F(A)\otimes'F(B)\rightarrow F(A\otimes B)$ natural in $A$ and $B$ ,
		\item[-] for the units morphism $m_I:I'\rightarrow F(I)$
	\end{itemize}
	which satisfy the following coherence axioms:
	\[
	\xymatrix{
		FA\otimes' (FB\otimes' FC)\ar[d]^{\alpha'} \ar[r]^{1\otimes' m}  & FA\otimes' F(B\otimes C)\ar[r]^{m}  & F(A\otimes (B\otimes C)) \ar[d]^{F\alpha}  \\
		(FA\otimes' FB)\otimes FC\ar[r]^{m\otimes' 1} &F( A\otimes B)\otimes' FC\ar[r]^{m} & F((A\otimes B)\otimes C)  }
	\]
	
	$$\begin{array}{cc}
	\xymatrix@=20pt{
		FA\otimes' I' \ar[r]^{\rho'}  \ar[d]_{1\otimes'm}& FA   \\
		FA\otimes' FI  \ar[r]_{m}                       &  F(A\otimes I) \ar[u]^{F\rho}
	} \hspace{3cm}&
	\xymatrix@=20pt{
		I'\otimes' FA \ar[d]^{m\otimes'1} \ar[r]^{\lambda'}  &  FA  \\
		FI\otimes'FA   \ar[r]_{m}        &  F(I\otimes A )\ar[u]^{F(\lambda)}
	}\end{array}$$
\end{definition}

A monoidal functor is \textit{strong} when $m_I$ and for every $A$ and $B$ $m_{A,B}$ are isomorphisms. It is said to be \textit{strict} when all the $m_{A,B}$ and $m_I$ are identities.\\


\begin{definition}
	\label{MONOIDAL NATURAL TRANFORMATION}
	\rm
	A \textit{monoidal natural transformation} $\theta:(F,m)\rightarrow (G,n)$ between monoidal functors is a natural transformation
	$\theta_A:FA\rightarrow GA$ such that the following axioms hold:
	$$\begin{array}{cc}
	\xymatrix@=20pt{
		FA\otimes' FB \ar[rr]^{m}\ar[d]_{\theta_{A}\otimes'\theta_{B}}
		&& F(A\otimes B)\ar[d]^{\theta_{A\otimes B}}\\
		GA\otimes' GB \ar[rr]_{n}
		&& G(A\otimes B)
	}\hspace{3cm}&
	\xymatrix@=25pt{
		I' \ar[r]^{m_I}\ar[rd]_{n_I}
		& FI\ar[d]^{\theta_I}\\
		& GI
	}\end{array}$$
\end{definition}

\begin{definition}
	\rm
	A \textit{monoidal adjunction}
	\[\xymatrix{
		(\cV,\otimes ,I)\ar@<1ex>[r]^{(F,m)}& (\cW,\otimes',I') \ar@<1ex>[l]^{(G,n)}_{\bot}}\]
	between two monoidal functors $(F,m)$ and $(G,n)$ consists of an adjunction $(F,G,\eta,\varepsilon)$ in which the unit $\eta:Id\Rightarrow G\circ F$ and the counit $\varepsilon:F\circ G\Rightarrow Id$ are monoidal natural tranformations, as defined~(\ref{MONOIDAL NATURAL TRANFORMATION}).
\end{definition}
\begin{proposition}\rm
	\label{Kelly}
	Let $(F,m):\cC \rightarrow\cC'$ be a monoidal functor. Then $F$ has a
	right adjoint $G$ for which the adjunction $(F,m)\dashv (G,n)$ is monoidal
	if and only if $F$ has a right adjoint $F\dashv G$ and $F$ is strong
	monoidal.
\end{proposition}
\begin{proof}\cite{kn:K}
\end{proof}

We recall that a \textit{cartesian category} is a category admitting finite products (products of a finite family of objects). Equivalently, a cartesian category is a category admitting binary products and a terminal object (the product of the empty set family of objects).

\subsection{Coalgebras and their comodules}\label{cosemisimple}

For general preliminaries on coalgebras we refer to \cite{kn:DR}. \\
\ \\
We work in the category $Vec$ of vector spaces over a fixed field $\Bbbk$, equiped with the usual tensor product over $\Bbbk$ that we denote by $\otimes$ and call $\tau: \otimes\Rightarrow \otimes^{op}$ the usual transposition, i.e. $\tau_{V,W}(v\otimes w)=w\otimes v$. We will write simply $\tau(v\otimes w)=w \otimes v$.\\
\  \\
 Explicitely, a $\Bbbk$-coalgebra is a data $(C,\Delta, \varepsilon)$ in $Vec$, where $\Delta: C\to C\otimes C$ (the {\em comultiplication} of the coalgebra) is coassociative and $\varepsilon: C\to \Bbbk$ is a {\em counit} for $\Delta$. This is dual to the notion of a $\Bbbk$-algebra. \\
 We use Sweedler's notation: if $(C,\Delta)$ is a coalgebra and $x\in C$, it is usual to denote
 $$
 \Delta(x)=\sum x_1\otimes x_2 \in C\otimes C.
 $$
 Coassociativity is expressed by $\sum (x_1)_1\otimes (x_1)_2\otimes x_2=\sum x_1 \otimes (x_2)_1\otimes (x_2)_2$ so this element is denoted by $\sum x_1\otimes x_2 \otimes x_3$.\\
 Counitality axioms are expressed by the equalities $\sum \varepsilon(x_1)x_2=x=\sum x_1\varepsilon(x_2)$.
 \\

 We assume during all the paper that $C$ is a cocommutative coalgebra over $\Bbbk$.

 The coalgebra being cocommutative means that $\Delta=\tau \circ \Delta$.  (Note that in Sweedler's notation, cocommutativity is expressed by $\sum x_1\otimes x_2=\sum x_2\otimes x_1$.) \\

Morphisms of coalgebras are $\Bbbk$-linear maps that preserve comultiplications and counits.\\
In Sweedler's notation, $f:C\to D$ is a morphism of coalgebras if and only if
$$
\sum f(x)_1\otimes f(x)_2=\sum f(x_1)\otimes f(x_2) \ \mbox{  and  } \sum \varepsilon_D \circ f=\varepsilon_C.
$$
\ \\
Dualising the notion of module over an algebra, we get that a left $C$-comodule is a pair $(V,v)$ in $Vec$ such that $v:V\to V\otimes C$ verifies the following two conditions:
\begin{itemize}
	\item $(id_V \otimes \Delta)\circ v=(v\otimes id_C)\circ v$,
	\item $( id_V\otimes \varepsilon)\circ v=id_V$ (we are using the identification $V \otimes \Bbbk\cong V$).
\end{itemize}
We say that $v$ is a {\em right coaction} on $V$.\\
\ \\

Many of the notions we deal with, have a left and a right version for general coalgebras. We assume all the time that $C$ is cocommutative, so there is no distinction.
In  particular, left comodules are also right comodules and viceversa so we will talk about $C$-comodules.\\
\ \\

Sweedler's notation is also used for comodules: for $x\in V$, we note
$$
v(x)=\sum x_0\otimes x_1 \in V\otimes C.
$$
The axioms of coaction give:
$$
\sum x_0\varepsilon(x_1)=x
$$
and
$$
\sum x_0\otimes (x_1)_1 \otimes (x_1)_2=\sum (x_0)_0\otimes (x_0)_1\otimes x_1
$$
and this term is written $\sum x_0 \otimes x_1\otimes x_2$.\\
\ \\

We denote by $Vec^C$ the category of $C$-comodules with morphisms $f:(V,v)\to (W,w)$ such that $f:V\to W$ is a $\Bbbk$-linear map such that $(f\otimes id_C)\circ v=w \circ f$.\\
In Sweedler's notation $f$ is a morphism of comodules if $\sum f(v)_0\otimes f(v)_1=\sum f(v_0)\otimes v_1$.
\subsubsection{Morphisms and cotensor product of comodules.}\label{ss:homyprod}
Let $(V,v)$ and $(W,w)$ be $C$-comodules, where $v$ and $w$ are right coactions.\\
\ \\
There is a structure of $C$-comodule denoted by $V\otimes^C W$ in the vector space generated by 
$\{x\otimes y \in V\otimes W\mid v(x)\otimes y=x\otimes \tau\left (w(y)\right )\}$ where the coaction is defined by $\delta(x\otimes y)=x\otimes w(y)$ (note that this last term is equal to $(id_V\otimes \tau) (v(x)\otimes y)$).\\
Of course, there is a natural definition of an endofunctor $(V\otimes^C-)$ of $Vec^C$ for each $C$-comodule $(V,v)$. We will see in Lemma \ref{cotensor} that $(Vec^C,\otimes^C,C)$ is a symmetric monoidal category. Moreover, it can be seen that $(V\otimes^C -)$ preserves coproducts but not necessarily colimits.\\
\ \\
We can also consider the vector space $Hom^C(V,W)$ of all morphisms of $C$-comodules from $V$ to $W$ that induces a covariant and a contravariant functor denoted by $Hom^C(V,-), Hom^C(-,W): Vec^C\to Vec$.

\subsubsection{Coflat and injective comodules}

This section deals with exactness of the functors mentioned above and motivates the necessity of dealing with cosemisimple coalgebras (see section \ref{ss:coss}). For more details we refer to \cite{kn:BW}.
\begin{definition}\rm
	An object in a (locally small) abelian category $\mathcal{C}$ is said to be \textit{injective}, if the contravariant functor $Hom_{\mathcal C}(-,V):\mathcal{C}\to Set$ takes monomorphisms into epimorphisms.
\end{definition}
\begin{remark}\rm \label{r:injective}
	\begin{enumerate}
		\item In the particular case in which the category ${\mathcal C}$ is $Vec^C$, injectivity of an object $V$ can be state as the right exactness of the functor $Hom^C(V,-): Vec^C\to Vec$.
		\item The contravariant functor $Hom^C(-,V):Vec^C\to Vec$  is always left exact. The comodule $V$ is injective as an object in $Vec^C$ if and only if the mentioned functor is exact.
		\item If $V$ is an injective $C$-comodule and $V\subseteq W$ in $Vec^C$ then $V$ is necessarily a direct summand of $W$.
	\end{enumerate}
\end{remark}
\begin{definition}\rm
	A comodule $V$ is said to be \textit{coflat} if the (covariant) functor $(V\otimes^C -): Vec^C \to Vec$ preserves epimorphisms.
\end{definition}
\begin{remark}\rm \label{r:coflat}
	\begin{itemize}
		
		\item The functor $(V\otimes^C-): Vec^C \to Vec$ is always left exact. The comodule $V$ is coflat if and only if the mentioned functor is exact.
		\item In particular, if $V$ is coflat, the functor $V\otimes^C-$ preserves all colimits.
		\item As $X \otimes^C C\cong X$ is a natural isomorphism in $X$, it is clear that $C$ is coflat as a $C$-comodule.
		\item A direct sum of coflat $C$-comodules is a coflat $C$-comodule, by distributivity of $\otimes^C$ over $\oplus$.
		\item It is easy to verify that if $V$ is a coflat $C$-comodule and $W$ is a direct summand of $V$ in $Vec^C$, then $W$ is also coflat.
	\end{itemize}
\end{remark}
The following is an important result in comodule theory (due to local finiteness) that will be crucial in our work. For a proof, see 10.12 in \cite{kn:BW} for the particular case in which $R=\Bbbk$.
\begin{proposition}\label{p:injcoflat}\rm
	Let $C$ be a coalgebra. A $C$-comodule is injective if and only if it is coflat.
\end{proposition}

\subsubsection{Cosemisimple coalgebras}\label{ss:coss}
\begin{definition}\rm
	A coalgebra is said to be {\em simple} if it does not have proper subcoalgebras (i.e. subcoalgebras other than $0$ or itself).\\ 
	A coalgebra is said to be {\em cosemisimple} if it is the direct sum of simple coalgebras.
\end{definition}
It is known (see for example \cite{kn:A}) that every comodule over a cosemisimple coalgebra is completely reducible. For general facts on abelian categories, this is equivalent to the fact that every short exact sequence splits and therefore every comodule is injective. So, for every comodule $W$ over a cosemisimple coalgebra, the functors  $Hom^C(-,W)$ and, using Proposition \ref{p:injcoflat}, $(-\otimes^C W)$ are exact.\\
These cosemisimple coalgebras are exactly the coalgebras for which the monoidal category $Vec^C$ is closed, as we will see in Lemma \ref{cotensor}.

The trivial coalgebra $\Bbbk$ with $\Delta (\lambda)=\lambda\otimes 1(=1\otimes \lambda), \varepsilon=id_\Bbbk$ is a cocommutative cosemisimple coalgebra. Indeed, it is obvious that $\Bbbk$ is cocommutative and, on the other hand, comodules over $\Bbbk$ are exactly $\Bbbk$-vector spaces ($Vec^\Bbbk=Vec$), which are always injective. \\
\ \\
We state now some important needed facts on cosemisimple coalgebras and their comodules. For a deeper approach and proofs we refer to \cite{kn:A}.\\
A coalgebra is said to be {\em simple} if it has no proper subcoalgebras. A comodule is said to be {\em simple} if it has no proper subcomodules. \\
Given a simple coalgebra $C$ there is only one simple $C$-comodule modulo isomorphisms. \\
A coalgebra version of Wedderburn's Theorem states that cosemisimple coalgebras are direct sum of simple ones and describes all simple coalgebras over a field $\Bbbk$. \\
In the particular case in which the coalgebra is cocommutative the description is simpler and it is easy to verify that the product of cocommutative cosemisimple coalgebras is also cosemisimple.

\section{A linear logic model} \label{s:lnl}
Intuitionistic Linear Logic (\cite{kn:G}) is a logic in which resources are controlled by an operator denoted with the simbol ``!" that expresses the notion of duplicability.\\
Translated to the model, this is captured by a monoidal comonad arising from a monoidal adjunction between a cartesian category and a symmetric monoidal closed category. There are many equivalent categorical descriptions of what should be such a model~\cite{kn:Me},~\cite{kn:Bi},~\cite{kn:S},~\cite{kn:Be}. The first categorical models of linear logic were given in \cite{kn:L} and \cite{kn:S}.  In this paper we follow Benton's ``Linear non linear category" definition.

\begin{definition}\rm
	A {\em LNL adjunction} consists of:
	\begin{itemize}
		\item a cartesian category $\mathcal{C}$,
		\item a closed symmetric monoidal category $\mathcal{M}$,
		\item a monoidal adjuntion $U\dashv R:\mathcal{C}\to \mathcal{M}$
	\end{itemize}
\end{definition}
\noindent
It is known that a LNL adjunction is a model for intuitionistic linear logic. We present here a family of examples of LNL adjunctions, constructed from cocommutative cosemisimple coalgebras over fields.\\
\ \\
Let $C$ be a cocommutative cosemisimple $\Bbbk$-coalgebra, where $\Bbbk$ is a field. We consider
\begin{itemize}
	\item the category $Coalg$ of $\Bbbk$-cocommutative coalgebras and morphisms of coalgebras;
	\item the category $Coalg C$ defined as follows:
	\begin{itemize}
		\item objects are morphisms of coalgebras with cocommutative codomain in $C$; we denote by $(\phi)$ the morphism of coalgebras $\phi:D\to C$ when it is thought as an object in $Coalg C$ (note that we ask $D$ to be cocommutative but not necessarily cosemisimple),
		\item if $\phi:D\to C$ and $\psi:E\to C$ are morphisms of coalgebras, morphisms $f:(\phi)\to (\psi)$ correspond to coalgebra morphisms $f:D\to E$ such that $\psi \circ f=\phi$;
	\end{itemize}
	\item the category $Vec^C$ of $C$-comodules;
	\item the functor $U^C: Coalg C \to Vec^C$ that takes the object $(\phi)$ with $\phi:D\rightarrow C$ to the comodule $(D,d)$ where $d:D\to D\otimes C$ is the coaction defined by $d=(id_D\otimes \phi) \circ \Delta_D$ ($U^C$ is defined on morphism in the obvious way). 
\end{itemize}
We will prove that $U^C$ admits a right adjoint $U^C\dashv R^C$ is an LNL adjunction.
\begin{lemma}\rm \label{cartesian}
	The category $Coalg$ of cocommutative coalgebras and morphisms of coalgebras is a cartesian category.
\end{lemma}
\begin{proof}
	Given two cocommutative coalgebras $(D_1, \Delta_1,\varepsilon_1)$ and $(D_2, \Delta_2, \varepsilon_2)$, the product coalgebra $D=D_1\times D_2$ is the cocommutative coalgebra defined on the vector space $D_1\otimes D_2$ by
	$$
	\begin{array}{l}
	\Delta(d_1\otimes d_2)=(id_{D_1} \otimes \tau \otimes id_{D_2}) (\Delta_1\otimes \Delta_2)(d_1\otimes d_2),\\
	\varepsilon (d_1\otimes d_2)=\varepsilon_1(d_1)\varepsilon_2(d_2),
	\end{array}
	$$
	where we denote by $\tau:D_1\otimes D_2 \to D_2\otimes  D_1$ the usual trasposition. \\
	Moreover, the trivial coalgebra $\Bbbk$ is a terminal object in the category of cocommutative coalgebras, so we have a cartesian structure on the category of cocommutative coalgebras.
\end{proof}
\begin{remark}\rm
	The operation $\times$ can be done for any two coalgebras, even if they are not cocommutative. More explicitely, if $C$ and $D$ are coalgebras, then $C\times D$ as defined above also is a coalgebra, but it will not be in general the cartesian product of $C$ and $D$. Now, if we ask $C$ and $D$ to be cocommutative, we get that $C\times D$ is the cartesian product of $C$ and $D$. This is based on the fact that for a cocommutative coalgebra $(D,\Delta,\varepsilon)$ the comultiplication $\Delta: D\to D\times D$ is a morphism of coalgebras. The corresponding projections are $p_D: \varepsilon_C \otimes id_D: C\otimes D \to D$ and $p_C:id_C \otimes \varepsilon_D: C\otimes D \to C$.
\end{remark}
\begin{lemma}\rm \label{coalgCcartesian}
	If $C$ is a cocommutative coalgebra, the category $Coalg C$ is a cartesian category.
\end{lemma}
\begin{proof}
	Note that $Coalg$ admits equalizers. Indeed, for a parallel pair $(f,g):D\to C$ it is enough to consider the largest subcoalgebra contained in $Ker (f-g)$.\\
	The existence of finite products (Lemma \ref{cartesian}) and equalizers in $Coalg$ guarantees the existence of pullbacks in this category, that induce a cartesian structure on $Coalg C$.\\
	To be explicit, we have that $(\phi_1)\times (\phi_2)=(\phi)$, where $\phi$ is defined by the following commutative diagram in $Coalg$, whose square is a pullback:
	$$
	\xymatrix
	{
		D\ar[rr]^-u \ar[dd]_-v \ar[rrdd]^\phi&& D_1\ar[dd]^-{\phi_1}\\
		\ \\
		D_2\ar[rr]_-{\phi_2}&&C.
	}
	$$
	where projection maps $\pi_1:(\phi_1)\times (\phi_2)\rightarrow (\phi_1)$ and $\pi_2:(\phi_1)\times (\phi_2)\rightarrow (\phi_2)$ are given by $\pi_1=u$ and $\pi_2=v$. Moreover, the terminal object is $(id_C)$.
	
\end{proof}
\begin{lemma}\rm \label{cotensor}
	If $C$ is a cocommutative coalgebra, the category $(Vec^C,\otimes^C,C)$
	is symmetric monoidal. Moreover, it is closed if and only if $C$ is cosemisimple.
\end{lemma}
\begin{proof}
	We already know the product $\otimes^C$ in $Vec^C$ (see~\ref{ss:homyprod}).  We can give a categorical characterization of the coaction associated with the product. For that, given $C$-comodules $(V, v)$ and $(W,w)$, if we note by $(V\otimes^C W,\rho_{V\otimes^C W})$ for the product $(V,v)\otimes^C (W,w))$, we get the following commutative diagram
	\begin{equation}\label{eq:splitfork}
	\xymatrix@C+5pt{
		V\otimes^C W \ar[r]^e   \ar[d]_{\rho_{V\otimes^C W}} &		V\otimes W \ar@<0.5ex>[rr]^-{id_V\otimes \tau w} \ar@<-0.5ex>[rr]_-{v\otimes id_W}\ar[d]_-{id_V\otimes w}&& V\otimes C \otimes W \ar[d]^-{id_V \otimes id_C \otimes w}\\
		(V\otimes^C W)\otimes C\ar[r]^{e\otimes id_C}&	V\otimes W \otimes C \ar@<0.5ex>[rr]^-{id_V\otimes \tau w\otimes id_C} \ar@<-0.5ex>[rr]_-{v\otimes id_W \otimes id_C} &&V\otimes C \otimes W\otimes C,
	}
	\end{equation}
	where both rows are equalizers. Indeed, using that the parallel pair in the first row is coreflexive via $id_V\otimes \varepsilon \otimes id_W$ (i.e. this last morphism is a common retraction for the morphisms of the parallel pair) and the fact that the functor $\_ \otimes C: Vec \to Vec$ preserves equalizers of coreflexive pairs, we obtain that the equalizer of the second row is $e\otimes id_C$.\\
	As the two involved squares commute, we get by universality property an induced map
	$$
	\rho_{V\otimes^C W}: V\otimes^C W \to V\otimes^C W \otimes C
	$$
	that defines a $C$-comodule structure on $V\otimes^C W$.  Associativity follows easily from associativity of $\otimes$ and coassociativity of $\Delta$. 
	(See \cite{kn:D} for associativity of $\otimes^C$). \\ It is also easy to check that $C$ is a unit.
	
The usual $\tau: V\otimes W \to W\otimes V$, $\tau(v\otimes w)=w\otimes v$ induces an isomorphisms between the equalizers $V\otimes ^C W$ and $W\otimes^C V$.
	\ \\
	Now, assume that $C$ is cosemisimple. This means that for every $C$-comodule $X$, the functor $(- \otimes^C X): Vec^C\to Vec^C$ is (left and) right exact. Also, $(-\otimes^C X)$ preserves all colimits (since it is an endofunctor on an abelian category preserving epimorphisms and coproducts), so, using the Special Adjoint Functor Theorem, we have that it has a right adjoint $hom^C (X,-)$ making $(Vec^C, \otimes^C, C, hom^C)$ a closed monoidal category.\\
	Conversely, if $Vec^C$ is monoidal closed, we have that for each $C$-comodule $V$, the functor $(V\otimes^C-)$ has a right adjoint, i.e., that $V$ is coflat, or equivalently, injective. Then, $C$ is cosemisimple.
\end{proof}
\noindent
We consider the forgetful functor $U: Coalg \to Vec$ taking each coalgebra to its underlying vector space.
Note that $U$ is an instance of $U^C$. Indeed, the functor $U^\Bbbk$ associated to the trivial coalgebra $\Bbbk$ is precisely $U$.
\begin{lemma}\rm \label{eq}\rm
	The functor $U:Coalg\to Vec$ preserves equalizers of coreflexive pairs.
\end{lemma}
\begin{proof}
	Let $f,g:C\to D$ be a coreflexive pair in $Coalg$ with equalizer $e:X\to C$. \\
	The equalizer in $Vec$ of $f$ and $g$ is the vector space $Ker(f-g)$. This is not in general the underlying vector space of the equalizer of $f$ and $g$ in $Coalg$. But, as $f,g$ have a common retraction $r$, it can be proved that $K:=Ker(f-g)$ is in fact a subcoalgebra of $C$, so the equalizer in $Coalg$ is $Ker(f-g)$ with its inclusion in $C$.\\
	Indeed, take $x\in Ker(f-g)$, that is $f(x)=g(x)$ and so $\Delta(f(x))=\Delta(g(x))$. As $f$ and $g$ are morphisms of coalgebras, we obtain
	$$
	\sum f(x_1)\otimes f(x_2)=\sum g(x_1)\otimes g(x_2).
	$$
	Applying $id \otimes r$ to the last equality and assuming the $x_2$'s are linearly independent, we get the equalities $f(x_1)=g(x_1)$, obtaining that $\sum x_1\otimes x_2 \in K\otimes C$. Similarly, we can prove that $\sum x_1\otimes x_2 \in C \otimes K$. We conclude that $\Delta(x)\otimes (K\otimes C) \cap (C \otimes K)= K\otimes K$.
\end{proof}
\begin{lemma}\rm \label{strong}
	The functor $U^C: Coalg C \to Vec^C$ is strong monoidal.
\end{lemma}
\begin{proof}
	It is clear that $U^C((id_C))=(C,\Delta_C)$, so $U^C$ preserves the units. We will prove now that $U^C((\phi)\times (\psi))=U^C(\phi) \otimes^C U^C(\psi)$. \\
	Take $\phi_1:(D_1, \Delta_1, \varepsilon_1) \to (C, \Delta_C, \varepsilon_C), \phi_2:(D_2, \Delta_2, \varepsilon_2)\to (C, \Delta_C, \varepsilon_C)$ be morphisms of coalgebras. We recall the pullback diagram defining the product $(\phi)=(\phi_1)\times (\phi_2)$:
	$$
	\xymatrix
	{
		D\ar[rr]^-u \ar[dd]_-v \ar[ddrr]^\phi && D_1\ar[dd]^-{\phi_1}\\
		\ \\
		D_2\ar[rr]_-{\phi_2}&&C.
	}
	$$
	Now, let $d=(id_D\otimes \phi)\circ \Delta$, $d_1=(id_{D_1}\otimes \phi_1)\circ \Delta_1$, $d_2=(id_{D_2}\otimes \phi_2)\circ \Delta_2$. \\
	Note that $U^C(D_i)=(D_i,d_i)$ for $i=1,2$ and $U^C(D)=(D,d)$.\\
	We will prove that $(D,d)=(D_1,d_1)\otimes^C (D_2,d_2)$, in other words that $D$-with a suitable morphism $d$- is the equalizer in $Vec$ of the following parallel pair and that $d$ is effectively $\rho_{D_1\otimes^C D_2}$ (with the notation of Lemma \ref{cotensor}). Consider 
	$$
	\xymatrix
	{
		D_1\otimes D_2  \ar@<-0.5ex>[rrrr]_-{(id_{D_1}\otimes \phi_1 \otimes id_{D_2})\circ (\Delta_1 \otimes id_{D_2})} \ar@<0.5ex>[rrrr]^-{(id_{D_1}\otimes \phi_2\otimes id_{D_2})\circ (id_{D_1}\otimes \Delta_2)}&&&&D_1\otimes C \otimes D_2.
	}
	$$
	First observe that the parallel pair above can be thought in $Coalg$. We will prove first that the coalgebra $D$-with the morphism of coalgebras $(u\otimes v)\circ \Delta: D \to D_1\otimes D_2$ is the equalizer in $Coalg$.  \\
	Indeed,
	$$
	\begin{array}{ll}
	(id_{D_1}\otimes \phi_1 \otimes id_{D_2})\circ (\Delta_1 \otimes id_{D_2})\circ (u\otimes v)\circ \Delta=\\
	\ \ \ \ = (id_{D_1}\otimes \phi_1 \otimes id_{D_2})\circ (u\otimes u \otimes v)\circ (\Delta \otimes id_D)\circ \Delta &\mbox{ since $u$ is a morphism of coalgebras}\\
	\ \ \ \ = (u\otimes (\phi_1 \circ u) \otimes v)\circ (\Delta \otimes id_D)\circ \Delta & \mbox{ by coassociativity of $\Delta$}.
	\end{array}
	$$
	and similarly
	$$
	\begin{array}{lr}
	(id_{D_1}\otimes \phi_2 \otimes id_{D_2})\circ (id_{D_1}\otimes \Delta_2 )\circ (u\otimes v)\circ \Delta=\\
\ \ \ \ \ = (id_{D_1}\otimes \phi_2 \otimes id_{D_2})\circ (u\otimes v \otimes v)\circ (id_D \otimes \Delta)\circ \Delta \\
\ \ \ \ \ = (u\otimes (\phi_2 \circ v) \otimes v)\circ (\Delta \otimes id_D)\circ \Delta
	\end{array},
	$$
	and, since $\phi_1 \circ u=\phi_2 \circ v$, then the morphisms above are equal.\\
	Now, assume we have a morphism of coalgebras $d':D'\to D_1\otimes D_2$ equalizing the parallel pair. If we consider $u'=p_1\circ d', v'=p_2\circ d'$, where $p_1$ and $p_2$ are the canonical projections, it is easy to see that $\phi_1 \circ u'=\phi_2 \circ v'$ and therefore, by universality of the pullback, there is a morphism of coalgebras $h:D'\to D$ such that $uh=u', vh=v'$. We then have
	$$
	\begin{array}{lll}
	(u\otimes v) \Delta  h&=(uh \otimes vh) \Delta' &\mbox{ since $h$ is a morphisms of coalgebras}\\ & =(u'\otimes v') \Delta &\mbox{ by definition of $h$}\\ &=(p_1 d' \otimes p_2 d')\Delta'& \mbox{ by definition of $u',v'$}\\ &=(p_1\otimes p_2) (id_{D_1}\otimes \tau \otimes id_{D_2}) (\Delta_1\otimes \Delta_2)  d'&\mbox{ since $d'$ is a morphism of coalgebras}\\
	&= d'.
	\end{array}
	$$
	\ \\
	This proves that $(D,(u\otimes v) \Delta)$ is the equalizer in $Coalg$ of the parallel pair above. Now, by Lemma \ref{eq}, we have that $U$ preserves equalizers of the coreflexive pairs. Then, we have that $(D,(u\otimes v) \Delta)$ is the equalizer in $Vec$ of the parallel pair above. (Note that the pair is coreflexive for $id_{D_1}\otimes \varepsilon_C \otimes {id_{D_2}}$ is a common retraction in $Coalg$.)\\
	\ \\
	It is easy to prove that $d$ is the desired coaction, i.e., that the following diagram commutes:
	$$
	\xymatrix
	{
		D \ar[rrr]^-{(u\otimes v) \circ \Delta} \ar[dd]_-d&&& D_1\otimes D_2 \ar[dd]^-{id_{D_1}\otimes d_2}\\
		\\
		D\otimes C \ar[rrr]^{\left ((u\otimes v) \circ \Delta\right )\otimes id_C}&&& D_1\otimes D_2 \otimes C
	}
	$$
\end{proof}

\begin{lemma}\rm \label{leftadjoint}
	The functor $U^C$ admits a right adjoint $R^C:Vec^C\to Coalg C$.
\end{lemma}
\begin{proof}
	We use the Special Adjoint Functor Theorem.\\
	We first show that $Coalg C$ is in the hypothesis of the Theorem. It is clear that $Coalg C$ is locally small. Moreover, colimits in $Coalg C$ are easily created by the colimits of the underlying coalgebras in $C$, which are created by the colimits of its underlying vector spaces in $Vec$. As $Vec$ is cocomplete, we get that $Coalg C$. \\
	Now, by the Fundamental Theorem of Coalgebras, we have that taking all finite dimensional cocommutative coalgebras we get a generating set of $Coalg$. It is easy to induce from this set, a generating set of $Coalg C$.\\
	Now, as $Coalg C$ is cowell-powered, $Vec^C$ is locally small and $U^C$ preserves colimits, we get that $U^C$ has a right adjoint that we call $R^C$.
\end{proof}
From the results above, we deduce the following
\begin{proposition} \rm If $C$ is a cocommutative cosemisimple coalgebra, then
	$$U^C \dashv R^C: Coalg C\rightleftarrows Vec^C$$
	is an LNL adjunction.
\end{proposition}
\begin{proof}
	Indeed, $Coalg C$ is a cartesian category (Lemma \ref{coalgCcartesian}), $Vec^C$ is monoidal closed (Lemma \ref{cotensor}) and $U^C \dashv R^C$ is a monoidal adjunction (Lemmas \ref{strong} and \ref{leftadjoint} and Proposition \ref{Kelly}).\\

\end{proof}


	
\section{A category indexed by coalgebras}

We keep restricting to the category $Coalg$ of cocommutative coalgebras over a fixed field $\Bbbk$ and we consider, for each such a coalgebra $C$, the category $Vec^C$ of its comodules. \\
This construction gives rise to what is known as an indexed category over $Coalg$ (we skip the general theory of indexed categories and only precise the needed notions in this context. For more details, we refer the reader to \cite{kn:PS}.)
Indeed,
\begin{itemize}
	\item for each object $C$ in $Coalg$, we have the category $Vec^C$,
	\item for each morphism $\phi: D\to C$ in $Coalg$, we get a functor $\phi^*:Vec^{C}\to Vec^{D}$, defined as follows: if $V:=(V,\rho^C:V\to V\otimes C)$ is a $C$-comodule, we consider
	$$
	\phi^*(V)=V\otimes^C U^C((\phi)),
	$$
	as underlying set but with a coaction induced by the one of $D$. We will be more precise in the following remark.
\end{itemize}

\begin{remark}\rm
	\begin{itemize}
		\item[1)] The elements in $\phi^*(V)$ are linearly generated by all elements of the form $v\otimes d$ verifying $\sum v_0\otimes v_1 \otimes d= \sum v\otimes \phi(d_1) \otimes d_2$ (see~\ref{ss:homyprod}).
		\item[2)] If we call $\rho$ the coaction of $\phi^*(V)$ we have that $\rho(v\otimes d)= \sum v\otimes d_1\otimes d_2$.
		\item[3)] If we call $\rho^C:V\to V\otimes C$ the coaction of $V$, then the pair $(\phi^*(V),\rho)$ can be described as the first column of the following diagram, where both rows are equalizers in $Vec$:
		$$
		\xymatrix
		{
			\phi^*(V) \ar[r] \ar[d]_\rho & V\otimes D \ar@<0.5ex>[rrr]^-{id_V\otimes ((\phi\otimes id_D)\Delta)} \ar@<-0.5ex>[rrr]_-{\rho^C\otimes id_D}\ar[d]_-{id_V\otimes \Delta}&&& V\otimes C \otimes D \ar[d]^-{id_V \otimes id_C \otimes \Delta}\\
			\phi^*(V)\otimes D \ar[r] & V\otimes D \otimes D \ar@<0.5ex>[rrr]^-{id_{V}\otimes ((\phi\otimes id_D)\Delta)\otimes id_D} \ar@<-0.5ex>[rrr]_-{\rho^C\otimes id_{D\otimes D}}&&& V\otimes C \otimes D \otimes D
		}
		$$
	\end{itemize}
\end{remark}
This indexed category will be essential in the construction of linear hyperdoctrines that we propose in Section \ref{s:lhyp}. We will work in fact by indexing $LNL$ adjunctions over coalgebras $C$ that will have $Vec^C$ as its symmetric monoidal closed underlying category, linked by functors of the form $\phi^*$.

\subsection{Beck-Chevalley condition}
The mentioned indexed category $(Coalg, Vec^{ \_}, \_^*)$ satisfies what is known as the {\em Beck (or Beck-Chevalley) condition}, meaning that $3$ of Proposition~\ref{BC} below holds. This is proved in \cite{kn:GP}; we present here a more explicit and selfcontained way to prove it.
\ \\
\begin{proposition}\rm \label{BC}
	\begin{enumerate}
		\item For every morphism $\phi:D\to C$, the functor $\phi^*$ has a left adjoint that we call $\Sigma_\phi$.
		\item If $\phi, \psi$ are composable morphisms of coalgebras, $\Sigma_{\psi \circ \phi}=\Sigma_{\psi} \circ \Sigma_\phi$, and therefore $(\psi\circ \phi)^*\cong \phi^*\circ \psi^*$.
	
		\item~\label{beck}
		Given a pullback diagram in Coalg
		$$
		\xymatrix{
			D\ar[r]^\delta \ar[d]_\gamma & D_1\ar[d]^\beta\\
			D_2 \ar[r]^\alpha & C
		}
		$$
		the canonical natural transformation
		$$
		\Sigma_\delta\gamma^* \Rightarrow \beta^*\Sigma_\alpha: Vec^{D_2} \to Vec^{D_1}
		$$
		is an isomorphism.
	\end{enumerate}
\end{proposition}
\begin{proof}
	\begin{enumerate}
		\item A morphism $\phi:D\to C$ of coalgebras has a natural way of transforming a $D$-comodule into a $C$-comodule. Indeed, define $\Sigma_\phi:Vec^D\to Vec^C$ by
		$$
		\Sigma_\phi (V,\rho^D)=(V,(id_V\otimes \phi)\rho^D).
		$$
		on objects and $\Sigma_\phi (f)=f$ on morphisms. \\
		It can be proved that $\Sigma_\phi$ is a left adjoint for $\phi^*$. Indeed, consider a $D$-comodule $(V,\rho^D)$ and a $C$-comodule $(W,\rho^C)$ and define a bijection
		$$
		Hom^C\left (\Sigma_\phi (V,\rho^D), (W,\rho^C)\right ) \leftrightarrow Hom^D\left ((V,\rho^D), \phi^*(W,\rho^C)\right )
		$$
		as follows: a morphism $f:V\to W$ of $C$-comodules corresponds to a morphism $\hat f: V\to \phi^*(W)$ of $D$-comodules defined by $\hat f(v)=\sum f(v_0)\otimes v_1$; a morphism $g:V\to \phi^*(W)$ of $C$-comodules corresponds to a morphism $\tilde g=(id\otimes \varepsilon_D)\circ g$ of $D$-comodules.
			\item The first part is clear by the definition of $\Sigma_\phi$ and the second one follows from the fact that $\phi^*$ and $\psi^*$ are respective right adjoints of $\Sigma_\phi$ and $\Sigma_\psi$ (using uniqueness of right adjoints up to natural isomorphisms).
		\item The unit of the adjunction $(\Sigma_\alpha, \alpha^*)$ is the natural transformation
		$$
		\eta_\alpha: 1_{Vec^{D_2}}\Rightarrow \alpha^*\Sigma_\alpha.
		$$
		Composing with $\gamma^*$ and using the commutation of the pullback diagram, we get a natural transformation
		$$
		\gamma^*\eta: \gamma^*\Rightarrow \delta^* \beta^* \Sigma_\alpha
		$$
		between functors from $Vec^{D_2}$ to $Vec^{D}$, which induces, by the counit of the adjunction $(\Sigma_\delta, \delta^*)$, the (canonical) natural transformation
		$$
		\Sigma_\delta \gamma^*\Rightarrow \beta^*\Sigma_\alpha
		$$
		between functors from $Vec^{D_2}$ to $Vec^{D_1}$.
		We want to prove that it is in fact an isomorphism and that it will make the following diagram commutative:	
		$$
		\xymatrix{
			Vec^{D_2}\ar[r]^{\Sigma_\alpha} \ar[d]_{\gamma^*} & Vec^C\ar[d]^{\beta^*} \ar@/_/@{=>}[dl]_\varphi \\
			Vec^{D} \ar[r]^{\Sigma_\delta} & Vec^{D_1}
		}
		$$	
		The explicit form of the natural morphism follows. Take a $D_2$-comodule $(V,\rho_2)$ (that we will call shortly $V$). Applying the functor $\beta^* \circ \Sigma_\alpha$ to $V$, we obtain the comodule that we will denoted by $(V_\alpha)^\beta$ defined on the vector space
		$$
		\Bbbk \{v\otimes d^1\in V\otimes D_1\mid \sum v_0\otimes \alpha(v_1)\otimes d^1=\sum v\otimes \beta(d^1_1)\otimes d^1_2\}
		$$
		by the coaction $\rho_1=id_V\otimes \Delta_{D_1}$.\\
		On the other hand, applying the functor $\Sigma_\delta \circ \gamma^*$ to $V$, we obtain the comodule that we will denoted by $(V^\gamma)_\delta$ defined on the vector space
		$$
		\Bbbk \{v\otimes d\in V\otimes D\mid \sum v_0\otimes v_1 \otimes d=\sum v \otimes \gamma(d_1)\otimes d_2\}
		$$ 
		equipped with the coaction $\tilde \rho_1 (v\otimes d)=\sum v\otimes d_1\otimes \delta (d_2)$.\\
		The instance in $V$ of the natural morphism in question is
		$$
		\varphi_V: (V_\alpha)^\beta \to (V^\gamma)_\delta,
		$$	
		defined by $\varphi (v\otimes d^1)=\sum v_0 \otimes t(v_1\otimes d^1)$, where $t:D'\to D$ is defined as follows: consider the equalizer in $Coalg$
		$\xymatrix{
			D'\hookrightarrow D_1\otimes D_2 \ar@<0.5ex>[rr]^-{\beta\circ \pi_1} \ar@<-0.5ex>[rr]_-{\alpha\circ \pi_2}&& C
		}
		$ and the (unique) morphism of coalgebras $p:D\to D_1\otimes D_2$ such that $\delta p=\pi_1$ and $\gamma p=\pi_2$. As $p$ equalizes the parallel map above, we deduce that there is a (unique) morphism of coalgebras $t:D'\to D$ such that $pt$ is the canonical map from $D'$ to $D_1\otimes D_2$.\\
		\ \\
		Finally, it can be proved that the inverse of $\varphi_V$ is the morphism $\psi_V:(V^\gamma)_\delta \to (V_\alpha)^\beta$,
		defined by $\psi_V (v\otimes d)=v\otimes \delta(d)$.
		
	\end{enumerate}
\end{proof}

\subsubsection{$\phi^*\dashv \forall_{\phi}$}

\begin{proposition}\rm

Let $\phi:D\rightarrow D'$ be a coalgebra map. The following propositions are equivalent:
\begin{itemize}
\item the $D'$-comodule $U^{D'}(\phi)=(D,(id \otimes \phi)\Delta)$ is coflat,
\item $\phi^*:Vec^{D'}\rightarrow Vec^D$ has a right adjoint $\forall_{\phi}:Vec^D\rightarrow Vec^{D'}$.
\end{itemize}
\end{proposition}

\begin{proof}

In order to obtain a right adjoint, note that $Vec^D$ is locally small, cocomplete and cowellpowered, for any coalgebra $D$, so, by the Special Adjoint Functor Theorem, it is enough to prove that $\phi^*$ preserves colimits.\\
It is easy to verify that, if $\phi:D\to D'$ is a morphism of coalgebras, the functor $\phi^*:Vec^{D'}\to Vec^D$ preserves coproducts (direct sums of comodules), so it remains to see that it also preserves coequalizers. As $Vec^D$ is abelian, it is enough to show that the $\phi^*$ preserves epimorphisms. Now, for a $
D'$-comodule $V$, we have
$$
\Sigma_\phi \phi^*(V)\cong V\otimes^{D'} U^{D'}(\phi),
$$
as $D'$-comodules. As $U^{D'}(\phi)$ is coflat by hypothesis, we get that $-\otimes^{D'} U^{D'}((\phi))$ preserves epimorphisms, and on the other hand it is clear that $\Sigma_\phi$ reflects epimorphisms, so we are done. \\
Conversely, if $\phi^*$ has a right adjoint, as $\Sigma_\phi$ also has, we get that $-\otimes^{D'}U^{D'}(\phi)$ has a right adjoint and therefore $U^{D'}(\phi)$ is coflat.
\end{proof}

\textit{Beck condition.}
It turns out that since we have $\sum_{\phi}\dashv \phi^{*}\dashv \forall_{\phi}$ and $\sum_{\phi}$ satisfies Beck  condition then by adjointness $\forall_{\phi}$ also  satisfies Beck condition  whenever it exists, i.e.,:

$$
	\xymatrix@=10pt{
		A\ar[rrrr]^{\vartheta}\ar[ddd]_{\phi}
		&&&& B\ar[ddd]^{\psi}\\
		&&&&\\
		&&&&\\
		C\ar[rrrr]_{\eta}
		&&&& D
	}\hspace{3cm}$$
	is a pullback then

$$
	\xymatrix@=10pt{
		Vect^{B}\ar[rrrr]^{\vartheta^{*}}\ar[ddd]_{\forall_{\psi}}
		&&&& Vect^{A}\ar[ddd]^{\forall_{\phi}}\\
		&&&&\\
		&&&&\\
		Vect^{D}\ar[rrrr]_{\eta^{*}}
		&&&& Vect^{C}
	}\hspace{3cm}$$
	
	commutes. See \cite{kn:GP} for details.

\subsection{More structure preserved}
The following Lemma is needed to prove monoidality of the functor $\phi^*$.

\begin{lemma}\label{l:frobenius}
Let $\phi: C \to D$ be a morphism of coalgebras. If $(V,\rho_V)$ is a
$C$-comodule and $(W,\rho_W)$ is a $D$-comodule (we will denote them by $V$ and $W$ respectively), then
$$
\Sigma_\phi(V\otimes^C \phi^*(W))\cong \Sigma_\phi(V)\otimes^D W.
$$
\end{lemma}
\begin{proof}
	It is enough to verify that the maps:
	$$
	\varphi:\Sigma_\phi(V\otimes^C \phi^*(W))\to \Sigma_\phi(V)\otimes^D W, \ \ \ \ 
	\psi: \Sigma_\phi(V)\otimes^D W \to \Sigma_\phi(V\otimes^C \phi^*(W))
	$$
	defined by $\varphi(v\otimes w \otimes c)=v\otimes w\varepsilon(c)$ and $\psi(v\otimes w)=\sum v_0\otimes w \otimes v_1$ are:
	\begin{itemize}
		\item indeed well defined,
		\item inverse to each other,
		\item morphisms of $D$-comodules.
	\end{itemize}
We leave the details for the reader. Note that it is enough to verify that $\varphi$ (or $\psi$) is a morphism of $D$-comodules.
\end{proof}
\begin{proposition}\rm \label{prop:ssmc}
	Let $\phi:C\to D$ be a morphism of cosemisimple coalgebras. Then the functor $\phi^*:Vec^{D}\to Vec^{C}$ is strong symmetric monoidal closed.
\end{proposition}
\begin{proof}
	In order to have strong monoidality, we need to prove that $\phi^*(V\otimes^{D} W)\cong \phi^*(V)\otimes^C \phi^*(W)$ as $C$-comodules and that $\phi^*(D)\cong C$ as $C$-comodules. The second isomorphism is obvious. The first one is induced from the following maps:
	$$
	\begin{array}{llll}
	V\otimes C\otimes W\otimes C \to V\otimes W\otimes C &&& V\otimes W\otimes C \to V\otimes C\otimes W\otimes C \\
	v\otimes  c\otimes w \otimes \tilde c \mapsto v \otimes w \otimes \varepsilon(c)\tilde c &&& v\otimes w\otimes c \mapsto \sum v\otimes c_1 \otimes w\otimes c_2.
	\end{array}
	$$
	It can be proved that these two morphisms induce inverses morphisms on the corresponding equalizers.\\
	It is easy to check that $\phi^*$ is symmetric.\\
	It remains to proved that $\phi^*$ is closed, i.e., that $hom^C(\phi^*(V),\phi^*(W))\cong \phi^* hom^{D}(V,W)$. Here we follow~\cite{kn:GP} Theorem 2.4. Let call $X$ the left term and $Y$ the right term of the equality we want to prove. We will see that
	$$
	Hom^C(Z,X)\cong Hom^C(Z,Y), 
	$$
	for all $C$-comodule $Z$. Indeed,
	
	$$
	\begin{array}{lll}
	Hom^C(Z,X)&\cong Hom^C (Z\otimes^C \phi^*(V),\phi^*(W) )&\mbox{by adjointness of $\otimes^C$}\\
	&\cong Hom^D \left (\Sigma_\phi (Z\otimes^C \phi^*(V)), W\right )&\mbox{by adjointness of $\phi^*$}\\
	& \cong Hom^D\left ( \Sigma_\phi(Z)\otimes^D V, W \right )&\mbox{by Lemma \ref{l:frobenius}}\\
	&\cong Hom^D \left (\Sigma_\phi(Z), hom^D(V,W)\right )&\mbox{adjointness of $\otimes^D$}\\
	&\cong Hom^C (Z,\phi^*(hom^D(V,W)))&\mbox{adjointness of }\phi^*
	\end{array}	
	$$
	
	The thesis follows by Yoneda's Lemma.
	
\end{proof}

\section{Linear Hyperdoctrine}\label{s:lhyp}

In this section, we recall the notion of Linear Hyperdoctrine presented in~\cite{kn:S3},~\cite{kn:S2},~\cite{kn:M} and we present some examples coming from coalgebras and comodules.\\
In order to give our definition of Linear Hyperdoctrine, we need to make some preliminary considerations.\\

We will use as ``codomain" of our linear hyperdoctrines, the category $LNL$ of the linear-non linear adjunctions.\\
In view of Lemmas $7$ and $12$ in \cite{kn:M}, the objects of $LNL$ are equivalent to what is known as linear categories. The following definition explicits what are the suitable morphisms between linear-non linear adjunctions, proposed in \cite{kn:M}.

\begin{definition}\rm
	The category {\bf $LNL$} has as objects the linear-non linear adjunctions.\\
	If $U\dashv R:\mathcal{C}\rightleftarrows \mathcal{S}$ and $U'\dashv R':\mathcal{C'}\rightleftarrows \mathcal{S'}$ are $LNL$ adjunctions, a morphism $(L,K): U\dashv R \to U'\dashv R'$ is a pair of functors $(L,K)$, where
	\begin{itemize}
		\item $L:\mathcal{C}\to \mathcal{C'},\ \  K:\mathcal{S}\to \mathcal{S'}$,
		\item $KU=U'L, \ \ LR=R'K$,
		\item $L$ is cartesian, $K$ is strong symmetric monoidal closed,
		\item For $\eta, \eta '$ the respective units of $U\dashv R, U'\dashv R'$, $L\eta=\eta'L$ (or, equivalently, $\varepsilon' K=K\varepsilon$, for $\varepsilon$ the counits of the adjunctions.)
	\end{itemize}
	We will call {\bf $LNL$} the category of linear-non linear adjunctions with such morphisms.\\
	
\end{definition}
We present a functor $\psi: Coalg^{op} \to LNL$ that will be the basis of our linear hyperdoctrine. For a cocommutative coalgebra $C$, $\psi(C)$ will be the $LNL$-adjunction associated to $C$ as presented in Section \ref{s:lnl}. We define now $\psi(f)$, for a morphism of coalgebras $f:C'\to C$, the pair of functors $\psi(f)=(L_f,K_f)$ where $K_f: Vec^C\to Vec^{C'}, L_f: Coalg C \to Coalg C'$ are defined by:

$$
K_f=f^*
$$
and the functor $L_f$ is defined:
\begin{itemize}
\item on objects by $L_f(\xymatrix{D\ar[r]^\phi & C})=(X,x),$
where $(X,x)$ is such that the following diagram is a pullback:
$$
\xymatrix
{
	X \ar[r]^x \ar[d]_{\tilde x} &C'\ar[d]^f\\
	D\ar[r]_\phi & C.
}
$$
\item $L_f$ is defined on morphisms by using universality properties of pullbacks.
\end{itemize}

We need the following lemma to prove the main result.

\begin{proposition}\rm
	The pair $(L_f, K_f)$ is a morphism of $LNL$-adjunctions.
\end{proposition}
\begin{proof}
	The picture is the following
	$$
	\xymatrix
	{
		Coalg C \ar[dd]_{L_f}\ar@<0.5ex>[rr]^{U^C} && Vec^C \ar@<0.5ex>[ll]^{R^C} \ar[dd]^{K_f}\\
		\\
		Coalg C' \ar@<0.5ex>[rr]^{U^{C'}} && Vec^{C'} \ar@<0.5ex>[ll]^{R^{C'}}
	}
	$$
	We know that $K_f$ is symmetric strong monoidal and closed by Proposition \ref{prop:ssmc}.\\
	Also, it is easy to check that the functor $L_f$ has a left adjoint given by $\tilde \Sigma_f ((\phi))=(f\circ \phi)$, so $L_f$ preserves all limits and therefore it is cartesian. \\
	\ \\
	Now, given $\phi:D\to C$ a morphism of coalgebras, let us calculate $K_fU^C((\phi))$ and $U^{C'}L_f(\phi)$.\\
	Denote $U^C((\phi))=(D,\rho_D)$ where $\rho_D:D\to D\otimes C$ is given by $\rho_D=(1\otimes \phi)\circ \Delta_D$. Applying $K_f$ to it, we obtain the following equalizer in $Vec^{C'}$:
	$$
	\xymatrix{
		K_f(U^C((\phi))\ar[r] & D\otimes C' \ar@<0.5ex>[rrr]^-{\rho_D \otimes 1} \ar@<-0.5ex>[rrr]_-{1\otimes (f\otimes 1)(1\otimes \Delta_{C'})}&&& D\otimes C\otimes C'
	}.
	$$
	
	On the other hand, $L_f(\phi)=x:X\to C'$ defined by the pullback of $\phi$ and $f$ as follows
	
	$$
	\xymatrix
	{
		 X \ar[r]^x \ar[d]_{\tilde x} &C'\ar[d]^f\\
		D\ar[r]_\phi & C
	}
	$$
	
 	and we get that $U^{C'}L_f((\phi))$ is the $C'$-comodule $(X,\rho_X)$ where $\rho_X: X\to X\otimes C'$ is defined by $\rho_X=(1\otimes x)\circ \Delta_X$.\\
	
	To see that both constructions are isomorphic, it is straightforward to prove that
	$$
	(\tilde x\otimes x)\circ \Delta_X: X\to D\otimes C'
	$$
	is the equalizer of the parallel pair
	
	$$
	\xymatrix
	{
		D\otimes C' \ar@<0.5ex>[rrr]^-{\rho_D \otimes 1} \ar@<-0.5ex>[rrr]_-{1\otimes (f\otimes 1)(1\otimes \Delta_{C'})}&&& D\otimes C\otimes C'
	}.
	$$	
	
	Now, if we take $\tilde \Sigma_f: Coalg {C'}\to Coalg C$ and $\Sigma_f: Vec^C\to Vec^{C'}$ the respective left adjoints of $L_f$ and $K_f$ (see Proposition \ref{BC}), it is easy to prove that $U^C\tilde{\Sigma}_f=\Sigma_f U^{C'}$ and therefore, by taking right adjoints we get $L_fR^C=R^{C'}K_f$.\\
	
	To prove that $\eta'_{L_f}=L_f(\eta)$ note that the commutation of the squares implies $$L_fR^CU^C(D,\phi)=R^{C'}U^{C'}L_f(V).$$ \\
	As $\eta_{(D,\phi)}=id_{(D,\phi)}$ and $\eta'_{(D',\phi')}=id_{(D',\phi')}$, the equality to prove is obvious. 

\end{proof}
\ \\
A linear hyperdoctrine will be some kind of `` indexed linear-non linear adjunction'' that satisfies the Beck-Chevalley condition among other properties.
\begin{definition}\rm
Let
	$$
	\Phi: {\mathcal B}^{op}\to LNL
	$$
be a functor where ${\mathcal B}$ is cartesian whose objects are generated as finite products of a single object $C$. \\
We fix the following notation:
	\begin{itemize}
		\item for each object $I$ in $\mathcal{B}$, the $LNL$-adjunction $\Phi(I)$ is denoted by $U^I \dashv R^I: \Phi_1(I)\rightleftarrows \Phi_2(I)$,
		\item for each morphism $f$ in $\mathcal{B}$, the morphism of $LNL$-adjunctions $\Phi(f)$ is denoted by the pair of functors $(\Phi_1(f),\Phi_2(f))$
		\end{itemize}
	We say that $\Phi$ is a {\em linear hyperdoctrine} if:
	\begin{enumerate}	
		\item for each object $I$ in $\mathcal{B}$,there are funtors $\exists_I, \forall_I:\Phi_2(I\times C)\to \Phi_2 (I)$ that are respectively left and right adjoints to $\Phi_2(\pi_I): \Phi_2(I)\to \Phi_2(I\times C)$, where $\pi_I: I \times C\to I$ is the canonical projection.
		\item for each morphism $f:J\to I$ in $\mathcal{B}$, the following diagram commutes
		$$
		\xymatrix
		{
			\Phi_2(I\times C)\ar[rr]^{\forall_I} \ar[dd]_{\Phi_2(f\times id_C)}&& \Phi_2(I)\ar[dd]^{\Phi_2(f)}\\
			\\
		\Phi_2(J\times C)\ar[rr]_{\forall_J} && \Phi_2(J)
		 }
		$$
		\item similar properties to $(1)$ and $(2)$ hold for projections $C\times I\to I$.
	\end{enumerate}
\end{definition}

Let $C$ be a cosemisimple cocommutative coalgebra and $\mathcal{B}$ be the cartesian complete subcategory of $Coalg$ whose objects are all coalgebras $C^{\otimes n}$, where $n$ is natural (note that $C^{\otimes 0}$ is the trivial coalgebra $\Bbbk$).
 \begin{theorem}\rm
 The functor
 $$
 \Phi: \mathcal{B}^{op} \to LNL
 $$
 that:
 \begin{itemize}
 	\item takes a coalgebra $D$ in $\mathcal{B}$ into its associated $LNL$ adjunction as described in Section \ref{s:lnl}, that is
 $$
 \Phi(D)=U^D \dashv R^D:
 Coalg D \rightleftarrows Vec^D
 $$
 \item takes a morphism $f: D\to D'$ in $\mathcal{B}$ into the morphism $\Phi(f)=(L_f, K_f)$ in $LNL$.
 \end{itemize}
 is a linear hyperdoctrine.
\end{theorem}
\begin{proof}
	Let us prove first condition (1).
We recall that
 $\Sigma_{\pi_I}$ is a left adjoint of $K_{\pi_I}$.\\

Note that, in fact, we proved that every $K_f$ admits a left and a right adjoint, in particular for $f=\pi_I$.\\
\ \\
Condition $(2)$ holds by taking right adjoints to the equality of the Chevalley-Beck condition proved in Proposition \ref{BC}.\\
\ \\
Condition $(3)$ is obvious by cocommutativity of $C$.

\end{proof}

\section{Towards a Model of Linear Polymorphism}

 When we interpret first order logic in an indexed category, we interpret terms as morphisms in the base, and predicates/formulas by the elements in the fibers. In second order logic, formulas and terms are not as clearly distinguished; we rather view formulas as particular terms (because we can substitute them for propositional variables).  Slightly simplifying, in the categorical model, substitution of terms into formulas is modeled by reindexing an element of a fiber (representing the formula) along a morphism in the base (representing the term). So, when we want to substitute one proposition by another, as we do in the forall-elimination rule, of second order propositional logic, we first have to transform one formula into a term. This is precisely what the phrase ``for all $A$ (representing a formula) in $\Psi(U^n)$ exists $f:U^n\rightarrow U$ (representing a term) such that $\Psi(f)(G)=A$" does in a generic condition.

In first order logic the syntax consists of (possibly sorts), terms, and formulas. This can be interpreted in an indexed category $\Phi:B^{op}\rightarrow Cat$ ($B$ with finite products) where sorts are interpreted as objects of $B$, terms are interpreted as morphisms in $B$, and formulas are interpreted by objects in an appropriate fiber. In second order propositional logic/polymorphism, we have no terms, but the formulas take their role. More precisely, the formulas take a ``dual role" functioning as terms and formulas as once. We see here a certain ``ambiguity"
between the notions of type, predicate, and term, of object and proof: a term
of type $A$ is a morphism into $A$, which is a predicate over $A$; a morphism $1\rightarrow A$ can be
viewed either as an object of type A or as a proof of the proposition A. 
Normally, we want formulas to be interpreted as objects in the fibers, so that we can talk about proofs as morphisms in the fibers. But we also want to be able to substitute one formula into another, and since substitution corresponds to reindexing (along morphisms in the base), we have to represent formulas also as morphisms in the base. Hence to each formula $\phi(X_1...X_n)$, we want to associate
\begin{itemize}
\item(1) an object in $\Psi(U^n)$, 
\item(2) a morphism of type $U^n\rightarrow U$
\end{itemize}
With the generic predicate, we can establish a correspondence between $(1)$ and $(2)$.
On one hand, given $f : U^n\rightarrow U$, $\Psi(f)(G)$ is an object of $\Psi(U^n)$. On the other hand, given an $X$ in $\Psi(U^n)$, the universal property of
the generic object gives an $f: U^n\rightarrow U$.

Tripos is the easiest example of this kind of situation. Since the indexed categories are really indexed preorders, and in this case the statement ``... and an
isomorphism ..." does not require a ``witness"  in a preorder viewed as a category, there is at most one isomorphism between any two objects.
If we omit the generic predicate from the definition of tripos, we get what is called a ``first order hyperdoctrine", which is a structure which
can model first order predicate logic. With a generic predicate, we can also model ``higher order logic", which means that we can interpret
quantification over truth values and power sets. The underlying set/object of the generic predicate is the type of ``truth values", and we can get ``power types" as exponential objects (assuming that the underlying category is cartesian closed).

Polymorphism can be viewed as a way of thinking about proofs in ``second order propositional logic". This is a smaller fragment of logic than what can be modeled in a tripos, but the difference is that we want to do it in a ``proof relevant" way, i.e. having a genuine indexed ``category" instead of an indexed preorder. 
Models of polymorphism are formalised by "$2\lambda\times$-hyperdoctrines" (see~\cite{kn:CROLE}). The
relation between triposes and such hyperdoctrines is that starting from a tripos, we restrict the underlying category to powers $Prop^n$ of $Prop$
(since we ``only" want to talk about propositions), but in exchange we want the fibers of the indexed category to be genuine categories, not only preorders. Via the propositions-types correspondence the system can be viewed as a programming language, but if we think about the types as sets then we obtain something that is against set theoretic intuitions, and causes "size problems". If we want to model polymorphism, the strength of polymorphism comes from the fact that we can substitute arbitrary formulas for propositional variables, and for this we need the generic predicate.


Therefore, in order to present a full model of Linear Polymorphism, we need to deal with the notion of a generic object in an indexed category. The thing is that it can be proved that there are no generic comodules over a cosemisimple cocommutative coalgebra. \\
We propose to avoid this constraint by considering smaller categories on each index. More precisely: instead of taking the category of all comodules $Vec^D$ in each index, we propose to take some apropriate full subcategories. The idea is to consider cosemisimple comodules whose decomposition into simple ones is such that there are at most $\kappa$ many simple of each isomorphism class (for some cardinal $\kappa$).\\
In some cases, we get positive answers. That is, there will be a generic objects in these  ``smaller'' indexed categories.\\
We think the ideas of last section adapted to these new context can give rise to models of Linear Polymorphism.

 \author Mariana Haim,\\
 {Centro de Matem\'atica, 
 	Facultad de Ciencias, \\
 	Universidad de la Rep\'ublica,\\
 	Montevideo, Uruguay.}\\
 {negra@cmat.edu.uy}\\

 \author Octavio Malherbe\\
 {
 	Departamento de Matem\'atica, Centro Universitario Regional, Este, Maldonado,\\Instituto de Matem\'atica y Estad\'\i stica Rafael Laguardia, Facultad de Ingenier\'\i a, Montevideo.\\ Universidad de la Rep\'ublica,\\
 	Uruguay.}\\
 {malherbe@fing.edu.uy}


\begin{thebibliography}{99}
	\bibitem{kn:A} Abella, A., {\em Cosemisimple coalgebras}, Annales des Sciences Math\'ematiques du Qu\'ebec, {\bf 30-2}, 2005.
	
	\bibitem{kn:AyL} Abramsky, S., Lenisa M., {\em Linear realisability and full completeness for typed lambda-calculi}, Annals of Pure and Applied Logic 134 (2005) 122-168.
	
	\bibitem{kn:Be} Benton, N.,  {\em A mixed linear and non-linear logic: Proofs, terms and models (extended abstract)}, Lectures Notes in Computer Science, {\bf 933}, 1994.
	
	\bibitem{kn:B} Borceux, F., {\em Handbook of Categorical Algebra}, Encyclopedia of Mathematics and Applications, {\bf 50}, 2008.
	
	\bibitem{kn:Bi} Bierman, G., {\em What is a categorical model of intuitionistic linear logic}, Lecture Notes in Computer Science, {\bf 902}, 1995.
	
		
    \bibitem{kn:BL} Block, R., Leroux, P. {\em Generalized dual coalgebras of algebras and applications to cofree coalgebras}, Journal of Pure and Applied Algebra, {\bf 36}, 1985.
	
	\bibitem{kn:BW} Brzezinski, T., Wisbauer, R.,  {\em Corings and comodules}, London Mathematical Society Lecture Notes Series, {\bf 309}, 2003.
	
	\bibitem{kn:CGW} Coquand, T., Gunter, C.A., Winskel, G. {\em Domain theoretic Models of Polymorphism}, Information and Computation, {\bf 81-2},1989.
	
	\bibitem{kn:CROLE} Crole, R., L.{\em Categories for types}, Cambridge University Press, {\bf 81-2},1994.
	
	
	\bibitem{kn:D} Doi, Y., {\em Homological coalgebra.}, Journal of the Mathematical Society of Japan, {\bf 33-1}, 1981.
	
	\bibitem{kn:DR} Dascalescu S., Nastasescu C., Raianu, S., {\em Hopf Algebras: an Introduction}, Pure and Applied Mathematics: A series of Monographs and Textbooks, 2000.
	
	\bibitem{kn:G} Girard, J.Y., {\em Linear Logic}, Theoretical Computer Science, {\bf 50-1}, 1987.
	
	\bibitem{kn:G2} Girard, J.Y. {\em The system F of variable types fifteen years later}, Theoretical Computer Science, {\bf 45}, 1986.
	
	\bibitem{kn:GP} Grunenfelder, L., Par\'e, R. {\em Families parametrized by coalgebras}, Journal of Algebra, {\bf 107}, 1987.
	
	\bibitem{kn:H} Hyland, J.M.E.,{\em The effective topos}, The L. E. J. Brouwer Centenary Symposium 165-216, 1982.
	
	
	\bibitem{kn:K} Kelly, G.M., {\em Doctrinal adjunction.}, Lectures Notes in Mathematics, Vol 420, 1974.

    \bibitem{kn:L} Lafont Y., {\em Logiques, cat\'egories et machines.}, PhD thesis, Universit\'e Paris 7, 1988.

	\bibitem{kn:M} Maneggia, P., {\em Models of Linear Polymorphism.}, thesis submitted to The University of Birmingham, School of Computer Science, for the degree of Doctor of Philosophy, 2004.
	
	\bibitem{kn:McL} Mac Lane, S., {\em Categories for the working mathematician.}, Graduate Texts in Mathematics, 1998.
	
	\bibitem{kn:Me} Melli\`es, P.A., {\em Categorical models of linear logic revisited.} Preprint, 2002.
	
	\bibitem{kn:PS} Par\'e, R., Schumacher, D., {\em Abstract Families and the Adjoint Functor Theorems.}, Lecture Notes in Mathematics, {\bf 661}, 1978.
	
	\bibitem{kn:Pi} Pitts, K.N. {\em Polymorphism is set theoretic, constructively}, Lectures Notes In Computer Science, {\bf 283}, 1987.
	
	\bibitem{kn:Pl} Plotkin, G., {\em Second order type theory and Recursion.}, Notes for Scott Fest, Unpublished manuscript, 1993. 
	
	\bibitem{kn:R} Reynolds, J.C., {\em Polymorphism is not set-theoretic}, Lectures Notes In Computer Science, {\bf 173}, 1984.
	
	
	
	\bibitem{kn:S2} Seely, R. {\em Categorical semantics for higher order polymorphic lambda calculus}, Journal of Symbolic Logic, {\bf 52-4}, 1987.
	
	\bibitem{kn:S} Seely, R., {\em Linear logic $*$-autonomous categories and cofree coalgebras}, In J. W. Gray and A. Scedrov, eds., Categories in Computer Science and Logic, {\bf 92} of Contemporary Mathematics, Amer. Math. Soc. 1989.
	
	\bibitem{kn:S3} Seely, R. {\em Polymorphic Linear Logic and Topos Model}  C.R. Math. Rep. Acad. Sci. Canada - Vol. {\bf XII}, No. 1, February 1990.
	
\end{thebibliography}
\end{document}